\documentclass[abstract=true]{scrartcl}
\usepackage{amsmath,amssymb,amsthm}
\usepackage{float}
\usepackage{biblatex}
\usepackage{hyperref}
\usepackage[a4paper]{geometry}
\newcommand{\GL}{\operatorname{GL}_2}
\newcommand{\PG}{\operatorname{PGL}_2}
\newcommand{\F}{\mathbb{F}}
\newcommand{\NR}{\mathcal{NR}}

\newcommand{\I}{\mathcal{I}}
\newcommand{\Tr}{\operatorname{Tr}}
\newcommand{\Aut}{\operatorname{Aut}}
\newcommand{\K}{\overline{K}\cup\{\infty\}}
\newcommand{\Stab}{\operatorname{Stab}}
\newcommand{\Fix}{\operatorname{Fix}}

\newtheorem{de}{Definition}
\newtheorem{bem}[de]{Remark}
\newtheorem{exam}[de]{Example}
\newtheorem{lem}[de]{Lemma}
\newtheorem{koroll}[de]{Corollary}
\newtheorem{theorem}[de]{Theorem}

\title{The Recursive Behavior of the Number of Irreducible Factors of Certain Rational Transformations over Finite Fields} 
\author{Max Schulz}

\begin{document}
	\maketitle
	\begin{abstract}
		Constructing irreducible polynomials over finite fields via rational transformations is an active field of study. In this paper, we show that the method can be used for more than that. By considering normalized rational functions $Q$ that generate an invariant subfield $\F_q(x)^G$ of $\F_q(x)$, for a subgroup $G\le\Aut(\F_q(x)/\F_q)$, and exploiting the factorization pattern of their rational transformations, we are able to prove that the sets arising from considering whether the rational transformation $f^Q$ has exactly $\deg(Q)$ irreducible factors of the same degree as $f$, or an irreducible factor of higher degree, admit a recursive behavior, i.e.\ their numbers satisfy a recursive equation. These sets occasionally can be described in terms of arithmetical properties the roots of the irreducible polynomials satisfy, e.g.\ their roots satisfy certain trace or power/non-power conditions. These can often be translated to properties of their minimal polynomials. With that, we prove a new theorem about the number of certain irreducible factors of a class of linearized polynomials, which gives a new context to a recent paper by \textsc{Granger}. 
		\end{abstract}
\section{Introduction}
Let $\I_2^n$ be the irreducible polynomials of degree $n$ over $\F_2$. Moreover, set $a_i(f)$ as the $i$-th coefficient of $f$, where $a_n(f)$ is its leading coefficient and $a_0(n)$ is its constant term. Let \[S_{a,b}(n):=\{f\in \I_2^n\mid a_{n-1}(f)=a,a_1(f)=b\}.\] \textsc{Granger} showed in \cite{granger} that \begin{equation}\label{granger1}
	|S_{1,1}(n)|-|S_{0,0}(n)|=\begin{cases}
		0,&\text{if } 2\nmid n\\
		|S_{1,\ast}(\frac{n}{2})|,&\text{if } 2\mid n,
	\end{cases}
\end{equation} where $S_{1,\ast}(\frac{n}{2})$ is the set of $f\in\I_2^{n/2}$ with $a_{n-1}=1$. The difference $|S_{1,1}(n)|-|S_{0,0}(n)|$ is the same as $|S_{1,\ast}(n)|-|S_{0,\ast}(n)|$. To see this note that \begin{align*}
	S_{1,\ast}(n)&=S_{1,0}(n)\cup S_{1,1}(n)\\
	S_{0,\ast}(n)&=S_{0,0}(n)\cup S_{0,1}(n)
\end{align*} and $|S_{1,0}(n)|=|S_{0,1}(n)|$ by taking the reciprocal $f(x)\mapsto x^nf(1/x)$, which is a bijection between these two sets (since we are in $\F_2$). Therefore \begin{align}\label{granger2}
	|S_{1,1}(n)|-|S_{0,0}(n)|&=(|S_{1,0}(n)|+|S_{1,1}(n)|)-(|S_{0,0}(n)|+ |S_{0,1}(n)|)\nonumber\\
	&=	|S_{1,\ast}(n)|-|S_{0,\ast}(n)|=\begin{cases}
		0,&\text{if } 2\nmid n\\
		|S_{1,\ast}(\frac{n}{2})|,&\text{if } 2\mid n.
	\end{cases}
\end{align}
What we find interesting about Equation \eqref{granger2} is that it has a recursive nature, which is not present in Equation \eqref{granger1}. Moreover, the sets $S_{1,\ast}(n)$ and $S_{0,\ast}(n)$ have a connection to a certain rational transformation. Let $Q=g/h\in\F_q(x)$ be a rational function and $F\in \F_q[x]$ be a polynomial, then the rational transformation of $F$ with $Q$ is \[F^Q(x)=\lambda_{g,h,F}h(x)^{\deg(F)}F\left(\frac{g(x)}{h(x)}\right),\] where $\lambda_{g,h,F}$ makes the polynomial monic.

Let $f\in\I_2^n$, then $f^{x^2+x}=f(x^2+x)$ is irreducible if and only if $a_{n-1}(n)=1$, hence \begin{align*}
	S_{1,\ast}(n)&=\{f\in\I_2^n\mid f(x^2+x)\text{ is irreducible}\}\\
	S_{0,\ast}(n)&=\{f\in\I_2^n\mid f(x^2+x)\text{ is not irreducible}\}.
\end{align*} Surprisingly enough, this interpretation gives rise to an alternative proof of Equation \eqref{granger2}, and equivalently also for Equation \eqref{granger1}. Explaining this proof, and applying it to other subsets of the irreducible monic polynomials over $\F_q$, is the main theme of this paper.

We quickly outline our method. Let $\F_q$ be the finite field with $q$ elements and $\I_q$ the set of irreducible monic polynomials over $\F_q$. If $G\le\PG(\F_q)\cong\operatorname{Aut}(\F_q(x)/\F_q)$, then $G$ does not only act on the rational functions, but also on a subset of $\I_q$ (that contains all the irreducible polynomials of degree greater than 1) by considering every element of $G$ as a rational function \[\left[\left(\begin{matrix}
	a&b\\
	c&d
\end{matrix}\right)\right] \leftrightarrow \frac{ax+b}{cx+d}\] and applying the rational transformation with this rational function. Now consider the fixed field $\F_q(x)^G$ of $G$, which contains all the rational functions invariant under the automorphisms in $G$. By Lüroth's Theorem, there exists a generator $Q\in \F_q(x)$ of degree $|G|$ such that $\F_q(x)^G=\F_q(Q)$. We can normalize $Q$ such that the numerator polynomial is of degree $|G|$, and the denominator polynomial is of lesser degree. Based on \cite{bluher1}, we call these rational functions \textit{quotient maps} for $G$. In \cite[Theorem 13]{matsch}, we showed that the rational transformation with quotient maps for a subgroup $G$ factors into irreducible polynomials, all of which have the same degree and belong to the same $G$-orbit under the action of ''rational transforming'' with elements out of $G$. This is our starting point. We show that, for a quotient map $Q$ of $G$, the map \begin{align*}
	\delta_Q:&\I_q\to \text{Set of }G\text{-orbits in }\I_q\\
	&F\mapsto G\text{-orbit, which are the factors of }F^Q
\end{align*} is a bijection, see Corollary \ref{bij2}. As a consequence of this bijection, we derive a new formula (see Theorem \ref{comb}) for the number of $G$-orbits in $\I_q^n$, the set of the irreducible monic polynomials of degree $n$. Combining this new formula with the standard formula for $G$-orbits, that is, grouping equivalence classes of the same size together, we obtain interesting formulae for certain subgroups of $\PG(\F_q)$, namely cyclic subgroups of prime order and elementary abelian subgroups. By going back from these $G$-orbits to their preimages under $\delta_Q$, where $Q$ is of degree $|G|=s$, we prove that the cardinality of the sets  \begin{align*}
	N(Q,n)&:=\{f\in\I_q^n\mid f^Q\text{ factors into exactly }s \text{ irreducible polynomials of degree } n \}\\
	Y(Q,n)&:=\{f\in\I_q^n\mid f^Q\text{ has an irreducible factor of degree greater than } n \}.\nonumber
\end{align*} admits a recursive behavior for the two types of subgroups mentioned above, see Theorem \ref{cyclicyn} and \ref{elemabelian}. Sometimes, these sets can be described in terms of arithmetical properties that the irreducible polynomials have to satisfy, e.g.\ the examples we mentioned in the beginning, that are, polynomials with $a_{n-1}=1$ belong to $Y(x^2+x,n)$ over $\F_2$, and those satisfying $a_{n-1}=0$ belong to $N(x^2+x,n)$. Hence, the recursive equations follow from specializing the subgroup $G$ and the quotient map, and then showing that the sets $N(Q,n)$ and $Y(Q,n)$ can be described in an alternative arithmetical way.

One of the most striking results of this paper is about the factorization of a class of linearized polynomials. Let $V\le_p\F_q$ be an $\F_p$-subspace of $\F_q$. The associated subspace polynomial is defined as \[Q_V(x)=\prod_{v\in V}(x-v),\] which is a $p$-linearized polynomial over $\F_q$. A result due to \textsc{Ore} \cite{orecontrib} states that, for every $n\in\mathbb{N}$, there is a polynomial $Q_V^{(n)}\in\F_q[x]$ such that \[Q_V\circ Q_V^{(n)}=x^{q^n}-x=Q_V^{(n)}\circ Q_V.\] Define \begin{align*}
	S_V(n)&:=\{f\in\I_q^n: f\nmid Q_V^{(n)}\}\\
	T_V(n)&:=\{f\in\I_q^n: f\mid Q_V^{(n)}\},
\end{align*} then these sets satisfy \begin{equation}\label{elemabelianintro}
	|S_V(n)|-(|V|-1)|T_V(n)|=\begin{cases}
		0,&\text{if }p\nmid n\\
		\frac{|V|}{p}|S_V(n/p)|,&\text{if }p\mid n
	\end{cases}
\end{equation} for all $n\in\mathbb{N}$. For $Q_V:=x^{p^m}-x$, where $\F_{p^m}\le\F_q$, it can be shown that \[Q_V(x)\circ \Tr_{q^n/p^m}(x)=x^{q^n}-x=\Tr_{q^n/p^m}(x)\circ Q_V(x),\] where $\Tr_{q^n/p^m}(x)$ is the trace polynomial from $\F_{q^n}$ to $\F_{p^m}$. Note that \begin{align*}S_V(n)&=\{f\in\I_q^n: f\nmid\Tr_{q^n/p^m}(x)\}\\&=\{f\in\I_q^n: \Tr_{q^n/p^m}(\alpha)\neq 0\text{ for a root }\alpha\text{ of }f\}\\&=\{f\in\I_q^n: -\Tr_{q/p^m}(a_{n-1}(f))\neq 0\}\end{align*} by the transitivity of the trace. Using an argument showing that the equations $\Tr_{q^n/p^m}(x)=a$ have the same number of solutions in $\{\alpha\in\F_{q^n}\mid \F_q(\alpha)=\F_{q^n}\}$ for all $a\in\F_{p^m}^{\ast}$, we get a generalization of \eqref{granger2}. Let \[S_a^m(n):=\{f\in\I_q^n\mid -\Tr_{q/p^m}(a_{n-1}(f))=a\}\] then \begin{equation}\label{gengranger}
	|S_a^m(n)|-|S_0^m(n)|=\begin{cases}
		0,&\text{if } p\nmid n\\
		p^{m-1}|S_a^m(n/p)|,&\text{if } p\mid n,
	\end{cases}
\end{equation} which becomes Equation \eqref{granger2} if $m=1$ and $q=2$. The proof will be explained in Subsection \ref{secuni}.

Apart from these generalized trace conditions, we derive recursive formulae for sets with power/non-power conditions. Let $s$ be a prime dividing $q-1$. The sets in question are \begin{align*}
	T_c(n)&:=\left\{f\in\I_q^n\mid  f(c)^{\frac{q-1}{s}}\neq 1\right\}\\
	U_c(n)&:=\left\{f\in\I_q^n\mid f(c)^{\frac{q-1}{s}}= 1\right\}
\end{align*} for $c\in\F_q$ and, for $u\neq v\in\F_q$, \begin{align*}
	C_{u,v}^s(n)&=\left\{f\in\I_q^n\mid f(u)^{\frac{q-1}{s}}\neq f(v)^{\frac{q-1}{s}}\right\}\\
	D_{u,v}^s(n)&=\left\{f\in\I_q^n\mid f(u)^{\frac{q-1}{s}}= f(v)^{\frac{q-1}{s}}\right\}.
\end{align*} Both exhibit the same recursive behavior, see Theorem \ref{powercond} and \ref{gensquarecond}. The latter theorem is first proved for $s=2$ in Theorem \ref{squarenoncond}, and then generalized by conjugating the corresponding cyclic subgroup that we used in the proof of Theorem \ref{powercond} and using ''the right'' equivalent rational function. See the last section for the details.

\section{Preliminaries}
Let $K$ be an arbitrary field and let $\F_q$ denote the finite field with $q$ elements, where $q$ is a power of the prime number $p$. The central object of this paper are the irreducible polynomials over $\F_q$. We denote the set of the monic irreducible polynomials over $\F_q$ as $\I_q$ (over an arbitrary field as $\I_K$), and the set of degree $n$ monic irreducible polynomials as $\I_q^n$, which we call the \textit{$n$-th degree level}. There is a natural group action of the linear rational functions over $\F_q$ on $\I_q$. Let $[A]=\frac{ax+b}{cx+d}\in\F_q(x)$ be a linear rational function with $ad-bc\neq 0$ and $f\in\I_q$ be such that $f(\frac{a}{c})\neq 0$, whenever $\frac{a}{c}$ is defined. The transformation \begin{equation}\label{grac}
	f^{[A]}(x):=\lambda_{A,f}(cx+d)^{\deg(f)}f\left(\frac{ax+b}{cx+d}\right),
\end{equation} where $\lambda_{A,f}$ is a normalizing factor that makes the resulting polynomial monic, is a right group action on the monic irreducible polynomials with $f(\frac{a}{c})\neq 0$ (see item 4 of Lemma \ref{basic}). One particularly nice property of this action is that it maps irreducible polynomials $f$ onto irreducible polynomials $f^{[A]}$ of the same degree, hence the action stays on the same degree level. There has been a renewed interest in this action and its invariants over finite fields since the papers by \textsc{Garefalakis} \cite{GarefalakisGL}, and \textsc{Stichtenoth} and \textsc{Topuzoğlu} \cite{sticht}. A polynomial $f\in\I_q$ is called $[A]$-invariant if $f^{[A]}=f$. The group of linear rational functions over a field $K$ under composition is isomorphic to the group $\PG(K)$, which consists of cosets of matrices $[A]=\{\alpha A\mid \alpha\in K^{\ast}\}$. This is why we chose the somewhat unusual notation $[A]$ for the linear rational function before. From now on, we use the symbol $[A]$ to denote the coset of \[A=\left(\begin{matrix}
	a&b\\
	c&d
\end{matrix}\right)\in\GL(\F_q)\] in $\PG(\F_q)$ and the linear rational function $\frac{ax+b}{cx+d}$ simultaneously. A polynomial $f\in\I_q$ is $G$-invariant, for a subgroup $G\le\PG(\F_q)$, if it is $[A]$-invariant for all $[A]\in G$. Note that $f$ being $[A]$-invariant is equivalent to $f$ being $\langle[A]\rangle$-invariant. \textsc{Reis} showed in \cite[Theorem 1.3]{reisEx} that irreducible polynomials of degree greater or equal to 3 can only have cyclic stabilizers in $\PG(\F_q)$, which is probably why polynomials invariant under arbitrary subgroups were never really considered before over finite fields. 

Here however, we have to be more general and extend the action to arbitrary polynomials. For that we need to look at the (general) Möbius transformation. Let $\K$ be the projective line on an algebraic closure of $K$. The group $\PG(K)$ acts on $\K$ via \[ [A](\gamma):=\frac{a\gamma+b}{c\gamma+d}.\] If the denominator is $0$, then $[A](\gamma)=\infty$. Moreover, $[A](\infty)=\frac{a}{c}$ and if $c=0$ then $[A](\infty)=\infty$. The rules for $\infty$ can be extended to apply for all rational functions, i.e.\ if $Q=\frac{g}{h}$ then \[Q(\infty)=\begin{cases}
	\infty,&\text{if }\deg(h)<\deg(g)\\
	\frac{\text{leading coefficient of }g}{\text{leading coefficient of }h},& \text{if }\deg(h)=\deg(g)\\
	0,& \text{if }\deg(h)>\deg(g).
\end{cases}\] For a subgroup $G\le\PG(K)$, we define \[G(\gamma)=\{[A](\gamma)\mid [A]\in G\}\] as the $G$-orbit of $\gamma\in\K$. Define \[\NR_K^G:=\{f\in K[x]\mid f\text{ monic and }f(\alpha)\neq 0 \text{ for all }\alpha\in G(\infty)\}\] and \begin{equation}\label{iqg}
	\I_K^G:=\I_K\cap \NR_K^G
\end{equation} as the irreducible monic polynomials in $\NR_K^G$. If $K=\F_q$ we simply write $q$ instead of $\F_q$, e.g. $\NR_q^G$. We have \begin{lem}[{\cite[Lemma 3]{matsch}}]\label{basic}
	Let $G\le\PG(K)$. For all $f,t\in\mathcal{NR}_K^G$ and $[A],[B]\in G$ the following hold: \begin{enumerate}
		\item $\deg(f^{[A]})=\deg(f)$.
		\item $f^{[AB]}=(f^{[A]})^{[B]}$ and $f^{[I_2]}=f$, so it is a right group action of $G$ on $\mathcal{NR}_K^G$.
		\item $(ft)^{[A]}=(f^{[A]})(t^{[A]})$.
		\item $f$ is irreducible if and only if $f^{[A]}$ is irreducible, thus the action can be restricted to $\I_K^G$.
	\end{enumerate}
\end{lem}
The roots of $G$-invariant polynomials are invariant under the Möbius transformation of $G$, as the next lemma shows. \begin{lem}\label{rootsginv}
	Let $f\in\I_K^G$ and $G\le\PG(K)$. Moreover, define $R_f:=\{\alpha\in\overline{K}\mid f(\alpha)=0\}$ and, for $[A]\in G$, we set \[[A](R_f):=\{[A](\alpha)\mid \alpha\in R_f\}.\] The polynomial $f$ is $G$-invariant if and only if $[A](R_f)=R_f$ for all $[A]\in G$.
\end{lem}
\begin{proof}
	The proof can be done as in Lemma 2.7.\ and Proposition 2.8.\ in \cite{sticht}. 
\end{proof}
We state a variation of \cite[Theorem 1.3]{reisEx} that we need later on in this paper.\begin{theorem}[{\cite[Theorem 20]{matsch}}]\label{galff}
	Let $f\in\F_q[x]$ be irreducible. If $f$ is $G$-invariant for a subgroup $G\le\PG(\F_q)$ and the set of roots of $f$ in $\overline{\F}_q$ only contains regular $G$-orbits under Möbius transformation, then $G$ has to be cyclic and $|G|\mid \deg(f)$. In particular, $\operatorname{Stab}_{\PG(\F_q)}(f)$ is cyclic for all irreducible polynomials $f$ of degree greater or equal to 3. 
\end{theorem}

A \textit{rational transformation} is defined as follows: Let $F\in\F_q[x]$ be a polynomial and $Q=\frac{g}{h}\in\F_q(x)$ be a rational function. We always assume rational functions to be represented as reduced fractions, i.e.\ $\gcd(g,h)=1$. The rational transformation of $F$ with $Q$ (or abbreviated $Q$-transform) is defined as 
\begin{equation}\label{rattrans}\centering F^Q(x)=\lambda_{g,h,F}h(x)^{\deg(F)}F\left(\frac{g(x)}{h(x)}\right)\end{equation}
where $\lambda_{g,h,F}$ makes the output-polynomial monic. Note that \eqref{rattrans} is the more general version of \eqref{grac}, since we allow any rational function, not only the linear ones. The degree of a rational function $Q=\frac{g}{h}$ is defined as $\deg(Q)=\max\{\deg(g),\deg(h)\}$. We collect the most important facts about rational transformations in the following lemma.
\begin{lem}[{\cite[Lemma 2]{matsch}}]\label{transbasic}
	Let $F,H\in K[x]$ be polynomials and $Q,R\in K(x)\setminus K$ be rational functions. The following hold \begin{enumerate}
		\item $\deg(F^Q)=\deg(F)\cdot\deg(Q)$ if and only if $F(Q(\infty))\neq 0$.
		\item $F^Q$ is the monic numerator of $F(Q(x))\in K(x)$.
		\item $(F\cdot H)^{Q}=F^Q\cdot H^Q$.
		\item If $F(Q(\infty))\neq 0$ and $F^Q$ is irreducible, then $F$ is irreducible.
		\item If $\gcd(F,H)=1$, then $\gcd(F^Q,H^Q)=1$.
		\item If $F(Q(\infty))\neq 0$, then $F^{Q\circ R}=(F^{Q})^R$.
		\item If $F(Q(\infty))=0$ and $R=s/t$, then there is $k\in\mathbb{N}$ and $a\in K^{\ast}$ such that $a t^k\cdot (F^Q)^R=F^{Q\circ R}$.
	\end{enumerate}
\end{lem}
This lemma subsumes Lemma \ref{basic}, except for the fourth item, by considering the elements of $\PG(K)$ as linear rational functions.

The rational functions that are most important in this paper are called \textit{quotient maps}. Let $G\le\PG(\F_q)$ be a subgroup, then, since $\PG(\F_q)\cong\operatorname{Aut}(\F_q(x)/\F_q)$, we get with Galois theory that $\F_q(x)/\F_q(x)^G$ is a Galois extension and $[\F_q(x)/\F_q(x)^G]=|G|$. By Lüroth's theorem, there is a rational function $Q=\frac{g}{h}$ of degree $\deg(Q)=|G|$ such that $\F_q(x)^G=\F_q(Q)$. This generator can be normalized such that the numerator polynomial has degree $|G|$ and the denominator polynomial has smaller degree. These rational functions are called \textit{quotient maps}.\begin{de}
	A quotient map, for the subgroup $G\le\PG(\F_q)$, is a rational function $Q_G=\frac{g}{h}\in\F_q(x)$ such that \begin{enumerate}
		\item $\F_q(x)^G=\F_q(Q_G)$, and
		\item $0\le\deg(h)<\deg(g)=|G|$.
	\end{enumerate}
\end{de}
\begin{bem}\label{abrem}
	If $Q_G,Q_G'\in\F_q(x)$ are two quotient maps for $G$, then there is $a\in\F_q^{\ast}$ and $b\in\F_q$ such that $Q_G'=aQ_G+b$. For the proof we refer the reader to \cite[Proposition 3.4]{bluher1}. 
\end{bem}
The name ''quotient map'' is explained by the following theorem. \begin{theorem}[{\cite[Proposition 3.9]{bluher1}}]\label{bij}
	The quotient map $Q_G$ induces a bijection between the set of $G$-orbits of $\overline{\F}_q\cup\{\infty\}$, denoted by $(\overline{\F}_q\cup\{\infty\})/G$, and $\overline{\F}_q\cup\{\infty\}$, more precisely \begin{align*}
		\psi:\begin{cases} (\overline{\F}_q\cup\{\infty\})/G\to \overline{\F}_q\cup\{\infty\},\\
			G(v)\mapsto Q_G(G(v))=Q_G(v)\end{cases}
	\end{align*} is a bijection and $\psi(G(\infty))=Q_G(\infty)=\infty$.
\end{theorem}
There are essentially three ways of how to calculate a quotient map given a subgroup $G\le\PG(\F_q)$ that we know of, see \cite{bluher1,Hou,guire}. In most of our proofs, we omit verifying that the rational functions we use are indeed quotient maps. However, it is not hard to check that the functions we claim to be quotient maps are invariant under the subgroups $G\le\PG(\F_q)$ and satisfy the degree condition.

The backbone of everything that follows is the next theorem. Let $G\le\PG(\F_q)$ and define \[r^G:=\{r^{[A]}\mid [A]\in G\}\] as the $G$-orbit of the polynomial $r\in\I_q^G$. We use the notation \[\prod r^G:=\prod_{t\in r^G}t\] to write the product of elements out of the same orbit in a more compact manner. Now onto the theorem: \begin{theorem}[{\cite[Theorem 13 and Lemma 22]{matsch}}]\label{main1}
	Let $F\in\mathcal{I}_q$ and $G\le\PG(\F_q)$ be finite with quotient map $Q_G$. Then there is $k\in\mathbb{N}$ and $r\in\mathcal{I}_q^G$ with $\deg(F)\mid\deg(r)$ such that \begin{equation*}
		F^{Q_G}=(\prod r^G)^k.
	\end{equation*} Moreover the exponent $k$ is equal to 1 if and only if a root $\gamma\in\overline{\F}_q$ of $F^{Q_G}$ is contained in a regular $G$-orbit under Möbius transformation, i.e.\ $|G(\gamma)|=|G|$.
\end{theorem} Note that if $k=1$, the degree of every irreducible factor of $F^{Q_G}$ is given by the formula \begin{equation}\label{degreeformula}
	\deg(r)\cdot |r^G|=\deg(F)\cdot |G|.
\end{equation}
\section{Combinatorics of Orbits of Irreducible Polynomials}
As a consequence of Theorem \ref{main1}, we obtain an important bijection for this paper. \begin{koroll}\label{bij2}
	Let $G\le\PG(\F_q)$ and $Q_G$ be a quotient map for $G$. Additionally, we denote by $\I_q^G/G$ the set of $G$-orbits in $\I_q^G$. The map $\delta_{Q_G}:\I_q\to \I_q^G/G$ with $F\mapsto r^G$, such that $F^{Q_G}=(\prod r^G)^k$, is a bijection.
\end{koroll}
\begin{proof}
	By Theorem \ref{main1}, $\delta_{Q_G}$ defines a mapping between $\mathcal{I}_q$ and $\mathcal{I}_q^G/G$. We show that $\delta_{Q_G}$ is surjective first. Let $r\in\mathcal{I}_q^G$ and $\gamma\in \overline{\F}_q$ be a root of $r$, so $r$ is the minimal polynomial of $\gamma$ over $\F_q$. Moreover, let $\alpha\in \overline{\F}_q$ be such that $Q_G(\gamma)=\alpha$, and denote by $F\in \mathcal{I}_q$ the minimal polynomial of $\alpha$. Then, $F^{Q_G}$ has $\gamma$ as a root, thus $r\mid F^{Q_G}$ and $F^{Q_G}=(\prod r^G)^k$ by Theorem \ref{main1}, hence $\delta_{Q_G}(F)=r^G$. 
	
	Now onto the injectivity: Let $F,H\in\mathcal{I}_q$ be such that $\delta_{Q_G}(F)=\delta_{Q_G}(H)=r^G$ for $r\in\mathcal{I}_q^G$, so $F^{Q_G}=(\prod r^G)^k$ and $H^{Q_G}=(\prod r^G)^l$. This means that both $F^{Q_G}$ and $H^{Q_G}$ have the same roots. We obtain: \begin{equation*}
		\bigcup\limits_{\alpha\in R_F}Q_G^{-1}(\alpha)=R_{F^{Q_G}}=R_{H^{Q_G}}=\bigcup\limits_{\beta\in R_H}Q_G^{-1}(\beta).
	\end{equation*} Therefore $R_F=R_H$, since $Q_G$ induces a bijection between $(\overline{\F}_q\cup\{\infty\})/G$ and $\overline{\F}_q\cup\{\infty\}$ by Theorem \ref{bij}. As $F$ and $H$ are irreducible, monic and share the same roots, they have to be equal. 
\end{proof}
The divisibility of $\deg(r)$ by $\deg(F)$ in Theorem \ref{main1} is not hard to prove. Let $F\in\I_q$ and $\gamma\in\overline{\F}_q$ be a root of $F^{Q_G}$. Consider an irreducible factor $r\mid F^{Q_G}$ that has $\gamma$ as a root. Notice that $\deg(F)\mid \deg(r)$ holds. To see this, recall that $Q_G(\gamma)$ is a root of $F$, hence $\F_q(Q_G(\gamma))\subseteq\F_q(\gamma)$ and the divisibility follows from the degree formula for field extensions.

Let $Q_G$ be a quotient map for $G\le\PG(\F_q)$. Define $d(Q_G)\in \mathbb{N}_0$ as the maximal number such that there exists an irreducible polynomial $F\in\I_q$ of degree $d(Q_G)$ with $F^{Q_G}=\prod(r^G)^k$ and $k>1$. If no such polynomial exists we set $d(Q_G)=0$. 

We explain why $d(Q_G)\le 2$. By Theorem \ref{main1} we know that $k\neq1$ if and only if a root $\gamma\in\overline{\F}_q$ of $F^{Q_G}$ is contained in a non-regular $G$-orbit under Möbius transformation, i.e.\ $|G(\gamma)|<|G|$. By standard group action arguments \[|G(\gamma)|<|G|\Leftrightarrow\text{ there exists }[A]\in G\text{ such that }[A](\gamma)=\gamma. \] The equation \[[A](\gamma)=\frac{a\gamma+b}{c\gamma+d}=\gamma\] leads to a polynomial equation of degree at most 2 in $\gamma$, hence $[\F_q(\gamma):\F_q]\le 2$. Let $r$ be the minimal polynomial of $\gamma$ over $\F_q$, then $F^{Q_G}=(\prod r^G)^k$ and, since $\deg(F)\mid \deg(r)$, we have $\deg(F)\le 2$. Consequently $d(Q_G)\le 2$.

The next step is proving that the number $d(Q_G)$ is independent of the choice of the quotient map.  
\begin{lem}\label{d(G)}
	Let $Q_G$ and $Q_G'$ be two quotient maps for $G\le\PG(\F_q)$, then $d(Q_G)=d(Q_G')$.
\end{lem}
\begin{proof}
	Let $a,b\in\F_q$ be such that $Q_G'=aQ_G+b$ as in Remark \ref{abrem}. Take $F\in\I_q$ so that \begin{equation}\label{k>1}
		F^{Q_G}=(\prod r^G)^k
	\end{equation} and $k>1$. We show that the rational transformation of $F^{[B]}$ with $Q'_G$, where \[B=\left(\begin{matrix}
		a^{-1}&-a^{-1}b\\
		0&1
	\end{matrix}\right)\in\GL(\F_q),\] is the same as $F^{Q_G}$ in \eqref{k>1}. For that we need to view $[B]$ as a rational function. Then, with Lemma \ref{transbasic} item 6, we get \begin{align*}
		(F^{[B]})^{Q_G'}=F^{[B]\circ Q_G'}=F^{Q_G}=(\prod r^G)^k
	\end{align*} with $k>1$. Note that we made use of the following calculation above \[[B]\circ Q_G'=a^{-1}Q_G'-a^{-1}b=a^{-1}(aQ_G+b)-a^{-1}b=Q_G.\] To finish the argument we claim the polynomial $F^{[B]}$ has the same degree as $F$. Indeed, $[B](\infty)=\infty$, so $F([B](\infty))\neq 0$ and Lemma \ref{basic} applies. This shows that the set of polynomials $F^{[B]}$, such that $F^{Q_G}=(\prod r^G)^k$ with $k>1$, is contained in the set of polynomials $H$ such that $H^{Q_G'}=(\prod t^G)^k$ with $k>1$. Repeating the argument with $Q_G=a^{-1}Q_G'-a^{-1}b$ and $[B^{-1}]$ proves the other inclusion. Hence, both sets essentially contain the same elements, because they are just translates of each other with the linear rational function $[B]$. Since transformations with $[B]$ do not affect the degree, we finish the proof.   
\end{proof}
\begin{de}
	For a subgroup $G\le\PG(\F_q)$ we define $d(G):=d(Q_G)$, where $Q_G\in\F_q(x)$ is an arbitrary quotient map for $G$. This is well-defined by the preceding lemma.
\end{de}

\begin{exam}\label{exam1}In this example we illustrate that all the values $0,1$ and $2$ are possible.
	\begin{itemize}
		\item Take $\F_q$ of characteristic $p$. The group \[G=\{x+i\mid i\in\F_p\}\] is cyclic of order $p$. A quotient map is $Q_G(x)=x^p-x$. To show that $d(G)=0$ we have to look at the non-regular $G$-orbits. By what we discussed before, we know that non-regular orbits correspond to elements that are invariant under elements out of $G$, hence we need to find solutions to the equation \[x+i=x.\] The sole solution to this equation is $\infty$, hence, by Theorem \ref{main1}, we know that $k$ is always 1, showing $d(G)=0$.
		\item Now take $\F_2$ and the group \[G=\left\{x,\frac{1}{x}\right\}.\] A quotient map is $Q_G(x)=\frac{x^2+1}{x}$. To show that $d(G)=1$ we first need to show that $d(G)<2$, which is easy since there is only one irreducible quadratic polynomial to check, namely $x^2+x+1$. Note that \[(x^2+x+1)^{Q_G}=x^4+x^3+x^2+x+1\] is irreducible, so $d(G)<2$. To show $d(G)\ge 1$ we calculate \[x^{Q_G}=x^2+1=(x+1)^2,\] so $k=2>1$.
		\item We stay in $\F_2$ and consider the group \[G=\left\{x,\frac{x+1}{x},\frac{1}{x+1}\right\}\] with quotient map $Q_G(x)=\frac{x^3+x+1}{x^2+x}$. To show that $d(G)=2$, we conveniently only have to look at \[(x^2+x+1)^{Q_G}=(x^2+x+1)^3.\]
	\end{itemize}
\end{exam}
Define \begin{equation}\label{omegank}\omega_G(n,k)=|\{r^G\mid r\in\I_q^G\cap \I_q^n\text{ and }|r^G|=k\}|\end{equation} as the number of $k$-sized $G$-orbits in $\I_q^n\cap\I_q^G$. Recall that $\I_q^n\cap\I_q^G=\I_q^n$ if $n>1$, since the condition that $f$ should not have roots in $G(\infty)$ only possibly excludes irreducible polynomials with roots in $\F_q$. 

The values of $\omega_G(n,k)$ for certain combinations of $n$ and $k$ are known. For example, \textsc{Reis} showed in \cite[Theorem 1.4.]{reisEx} that $\omega_G(n,1)=0$ if $G$ is cyclic and $|G|\nmid n$. Note that $\omega_G(n,1)$ is the number of $G$-invariant irreducible monic polynomials of degree $n$ and their explicit numbers are given in \cite[Theorem 1.4.]{reisEx}. In \cite[Theorem 1.3.]{reisEx}, it is shown that $\omega_G(n,1)=0$ if $G$ is non-cyclic and $n>2$.

As a consequence of the orbit-stabilizer theorem we additionally have that $\omega_G(n,k)=0$ if $k\nmid |G|$ for any group $G$. Since orbits are equivalence classes under the relation $f\sim g$ if and only if there is a $[A]\in G$ such that $f=g^{[A]}$, we can partition the sets $\I_q^G\cap \I_q^n$ into $G$-orbits\footnote{Under sums, when the condition is $k\mid |G|$, we write $\#G$ instead of $|G|$ because it looks better.}: \begin{equation}\label{part1}
	|\I_q^G\cap\I_q^n|=\sum\limits_{k\mid \#G}k\cdot \omega_G(n,k).
\end{equation} 
Another way to count $|\I_q^n|$ is \begin{theorem}
	\label{comb}
	Let $G\le\PG(\F_q)$. For all $m>d(G)$ we have \begin{equation*}
		|\I_q^m|=\sum\limits_{k\mid \#G}\omega_G(m\cdot \frac{|G|}{k},k).
	\end{equation*}
\end{theorem}
\begin{proof}
	Define \begin{equation*}
		M_G(m):=\{(a,b)\in\mathbb{N}^2:~b\mid |G| \text{ and } a\cdot b=|G|\cdot m\}
	\end{equation*} and \begin{equation*}
		\operatorname{Orb}_G(n,k):=\{r^G\mid r\in\I_q^n\cap\I_q^G\text{ and }|r^G|=k\}.
	\end{equation*}
	We want to show that \begin{equation*}
		\delta_{Q_G}(\mathcal{I}_q^m)=\bigcup_{(n,k)\in M_m(G)}\operatorname{Orb}_{G}(n,k)
	\end{equation*} for $m>d(G)$ and $\delta_{Q_G}$ as in Corollary \ref{bij2}. Let $F\in\I_q^m$, then $F^{Q_G}=\prod r^G$ for $r\in\I_q^G$, since $m>d(G)$. Therefore $$\deg(r)\cdot |r^G|=\deg(F)\cdot |G|$$ by comparing the degree of both sides, thus $(\deg(r),|r^G|)\in M_G(m)$ which shows \begin{equation*}
		\delta_{Q_G}(\mathcal{I}_q^m)\subseteq \bigcup_{(n,k)\in M_m(G)}\operatorname{Orb}_{G}(n,k).
	\end{equation*}
	For the converse, let $r^G\in\operatorname{Orb}_G(n,k)$ with $(n,k)\in M_G(m)$, hence $\deg(r)=n$, $|r^G|=k$ and $\deg(r)\cdot| r^G|=m\cdot |G|$. By Corollary \ref{bij2} there exists a polynomial $F\in\I_q$ such that $F^{Q_G}=(\prod r^G)^l$. As a result \begin{align*}
		\deg(F^{Q_G})&=|G|\cdot \deg(F)\\&=l\cdot\deg(r)\cdot |r^G|=l\cdot(|G|\cdot m),
	\end{align*} thus $\deg(F)=l\cdot m\ge m>d(G)$ and therefore $l=1$ by the definition of $d(G)$. With that we have \begin{equation*}
		\delta_{Q_G}(\mathcal{I}_q^m)\supseteq \bigcup_{(n,k)\in M_m(G)}\operatorname{Orb}_{G}(n,k).
	\end{equation*} Since $\delta_{Q_G}$ is a bijection, we get \begin{align*}
		|\I_q^m|&=|\delta_{Q_G}(\I_q^m)|=|\bigcup_{(n,k)\in M_G(m)}\operatorname{Orb}_G(n,k)|=\sum\limits_{(n,k)\in M_G(m)}|\operatorname{Orb}_G(n,k)|\\&=\sum\limits_{(n,k)\in M_G(m)}\omega_G(n,k)=\sum\limits_{k\mid \#G}\omega_G(m\cdot \frac{|G|}{k},k).
	\end{align*} 
\end{proof}
\begin{exam}
	Consider the group $G:=\Stab_{\PG(\F_3)}(-1)\le\PG(\F_3)$, which is of order $6$ and is generated by the two elements $$\left[\begin{pmatrix}
		0&1\\2&1
	\end{pmatrix}\right], \left[\begin{pmatrix}
		0&1\\1&0
	\end{pmatrix}\right].$$ Note that this group is dihedral. A quotient map is given by $$Q_G=\left(\frac{(x+1)^3}{x(x-1)}\right)^2.$$ The two Tables \ref{f31} and \ref{f32} include the factorizations of the $Q_G$-transformations of all the monic irreducible polynomials of degree $1$ and $2$ over $\F_3$.	\begin{table}[h]
		\centering
		\begin{tabular}{|c|c|}
			\hline
			$F$ & Factors of $F^{Q_G}$\\
			\hline
			$x$&$x+1$\\
			\hline
			$x+1$&$x^2+1$\\
			& $x^2+x+2$\\
			&$x^2+2x+2$\\
			\hline
			$x+2$& $x^3 + x^2 + 2x + 1$\\
			&$x^3 + 2x^2 + x + 1$\\
			\hline
		\end{tabular}	
		\caption{Factorization of the $Q_G$-transformations of degree 1 irreducible polynomials}
		\label{f31}
		\begin{tabular}{|c|c|}
			\hline
			$F$ & Factors of $F^{Q_G}$\\
			\hline
			$x^2+1$&$x^6+x^5+x^4+x^3+2x+1$\\
			&$x^6+2x^5+x^3+x^2+x+1$\\
			\hline
			$x^2+x+2$& $x^4+x ^3+x^2+x+1$\\
			&$x^4+x^3+x^2+2x+2$\\
			&$x^4+x^3+2x^2+2x+2$\\
			\hline
			$x^2+2x+2$& $x^4+x^2+x+1$\\
			&$x^4+x^3+x^2+1$\\
			&$x^4+2x^3+x^2+2x+1$\\
			\hline
		\end{tabular}
		\caption{Factorization of the $Q_G$-transformations of degree 2 irreducible polynomials}
		\label{f32}
	\end{table}

	Table \ref{f31} and \ref{f32} show that $d(G)=1$, because of the polynomial $F=x$. The formula in Theorem \ref{comb} for $m=2$ is \[3=N_3(2)=\omega_G(12,1)+\omega_G(6,2)+\omega_G(4,3)+\omega_G(2,6).\] The number $\omega_G(12,1)$ has to be $0$ because of the fact that there are no $G$-invariant irreducible polynomials of degree greater than 2 if $G$ is non-cyclic. Moreover, $\omega_G(2,6)=0$ as well, because there are only $3$ monic irreducible polynomials of degree $2$ over $\F_3$. In Table \ref{f32}, we see that $\omega_G(4,3)=2$ and $\omega_G(6,2)=1$.
	
	Let us verify that the number $\omega_G(4,3)=2$ is correct. It is well-known that \[|\I_q^n|=\frac{1}{n}\sum\limits_{d\mid n}\mu(d)q^{\frac{n}{d}},\] where $\mu(\cdot)$ is the Möbius function. Hence, there are exactly $\frac{1}{4}(3^4-3^2)=18$ irreducible degree 4 polynomials. The remaining $18-2\cdot 3=12$ polynomials are contained in two $G$-orbits of size 6, namely
	
	\begin{align*}
		(x^4+x^2+2)^G=\{&x^4+x^2+2,x^4+x^2+2x+1,x^4+2x^2+2,x^4+2x^3+x^2+1,\\&x^4+2x^3+x^2+x+2,x^4+2x^3+2x^2+x+2\}\\
		(x^4+x+2)^G=\{&x^4+x+2,x^4+2x+2,x^4+x^3+2,x^4+x^3+2x+1,x^4+2x^3+2,\\&x^4+2x^3+x+1\},
	\end{align*}
	showing that $\I_3^4$ decomposes into exactly four $G$-orbits, two of size 6 and two of size 3. As a side-note, \begin{align*}
		(x^4+2x^3+x^2+1)^{Q_G}&=\prod (x^4+x^2+2)^G\\
		(x^4+x^3+2x+1)^{Q_G}&=\prod (x^4+x+2)^G.
	\end{align*}
\end{exam}
If $\I_q^G\cap\I_q^n=\I_q^n$ and $n>d(G)$, we can combine equation \eqref{part1} and Theorem \ref{comb} \begin{align}\label{eqcombination}
	0=\sum\limits_{k\mid \#G}\left(\omega_G(n\cdot\frac{|G|}{k},k)-k\omega_G(n,k)\right).
\end{align} Recall that $\I_q^G\cap\I_q^n=\I_q^n$ if $n>1$ for any subgroup $G\le\PG(\F_q)$. Moreover, if $d(G)=0$, then we know that all $G$-orbits in $\overline{\F}_q\setminus G(\infty)$ are regular by Theorem \ref{main1}. Any element $[A]\in\PG(\F_q)$ has at least one fixed point in $\F_{q^2}\cup\{\infty\}$. Hence, $G(\infty)$ is the only non-regular orbit. 

Equation \eqref{eqcombination} restricts to simple formulae for subgroups $G$ with easy to describe cyclic subgroups of prime order. That is our goal now. For that purpose, we study the non-regular orbits of cyclic subgroups of $\PG(\F_q)$. Let $[A]\in\PG(\F_q)$ and \[\Fix([A]):=\{\gamma\in\overline{\F}_q\cup\{\infty\}\mid [A](\gamma)=\gamma\}.\] As said before, we have $1\le|\Fix([A])|\le 2$ and every $[A]$ belongs to one of the following three types: \begin{itemize}
	\item $[A]$ is of type 1 if it fixes a unique element in $\F_q\cup\{\infty\}$. Such an element has order $p$.
	\item $[A]$ is of type 2 if it fixes two distinct elements in $\F_q\cup\{\infty\}$. Such an element has order $d\mid q-1$.
	\item $[A]$ is of type 3 if it fixes two conjugated\footnote{$\alpha,\beta\in\F_{q^2}\setminus\F_q$ are conjugated if $\beta=\alpha^q$ or $\alpha=\beta^q$.} elements in $\F_{q^2}\setminus\F_q$. Such an element has order $d\mid q+1$.
\end{itemize} For reference see \cite[Proposition 11.1.]{bluher1}. For $G\le\PG(\F_q)$ define \[P_G:=\{\gamma\in\overline{\F}_q\cup\{\infty\}\mid |G(\gamma)|<|G|\},\] which is the set of elements that are contained in the non-regular $G$-orbits. The following lemma holds. \begin{lem}\label{fix}
	Let $[A]\in\PG(\F_q)\setminus\{[I_2]\}$ and $G=\langle [A]\rangle$. For all $z\in\mathbb{Z}$, such that $[A]^z\neq [I_2]$, we have $\Fix([A]^z)=\Fix([A])$. In particular, $\Fix([A])=P_G$.
\end{lem}
\begin{proof}
	The inclusion $\Fix([A]^z)\supseteq\Fix([A])$ is obvious. As mentioned before this lemma, we know $1\le|\Fix([A]^z)|\le 2$ for all $z\in\mathbb{Z}$ such that $[A]^z\neq [I_2]$. If $|\Fix([A]^z)|=|\Fix([A])|$, there is nothing to prove. Therefore, assume $|\Fix([A]^z)|=2$ and $|\Fix([A])|=1$ with $\alpha\in\Fix([A])$. Let $\beta\in\Fix([A]^z)\setminus\Fix([A])$, then \[[A]^z([A](\beta))=[A]([A]^z(\beta))=[A](\beta).\] Thus $[A](\beta)$ is fixed point of $[A]^z$ as well. Since $\gamma\mapsto [A](\gamma)$ is a bijection on $\overline{\F}_q\cup\{\infty\}$, we know that $[A](\beta)\neq \alpha$, which results in $[A](\beta)=\beta$, contradicting $|\Fix([A])|=1$. 
	
	Recall that $\gamma\in P_G$ if and only if there is $[B]\in G$ such that $[B](\gamma)=\gamma$. By the fact that $\Fix([A])=\Fix([B])$, for all $[B]\in G\setminus{[I_2]}$, we get the desired result $\Fix([A])=P_G$.
\end{proof}
The next lemma is a convenient fact which simplifies the proof of Theorem \ref{cyclicyn}. \begin{lem}\label{n=d(G)}
	Let $G\le\PG(\F_q)$ be a cyclic subgroup. If $P_G\neq\{\infty\}$, then $[\F_q(P_G\setminus\{\infty\}):\F_q]=d(G)$.
\end{lem}
\begin{proof}
	First let $[\F_q(P_G\setminus\{\infty\}):\F_q]=1$, thus $|P_G\cap\F_q|\ge 1$ and $G$ is generated by an element of type 1 or 2. Let $f$ be the minimal polynomial of $a\in P_G\cap\F_q$, then $f=x-a$ and there is a polynomial $F\in\I_q$ such that $F^{Q_G}=f^k$ by Theorem \ref{main1} and Corollary \ref{bij2}. Since $\deg(F)\mid \deg(f)$, we know $\deg(F)=1=\deg(f)$. As a consequence $d(G)=1$.
	
	Now let $[\F_q(P_G\setminus\{\infty\}):\F_q]=2$, then $G$ is generated by an element of type 3 and has two conjugated fixed points $\alpha,\alpha^q\in\F_{q^2}\setminus \F_q$. Let $f\in\I_q^2$ be the minimal polynomial of $\alpha$ (and $\alpha^q$) over $\F_q$. Assume, there is a degree 1 polynomial $F$ such that $F^{Q_G}=f^k$. The polynomial $F$ can be written as $F=x-a$ for some $a\in\F_q$. Then we get that \[Q_G(\alpha)=a=Q_G(\alpha^q).\] By Theorem \ref{bij}, the roots $\alpha$ and $\alpha^q$ would be contained in the same $G$-orbit, contradicting them being fixed points. Hence $d(G)=2$.
\end{proof}
Another family of important subgroups are the elementary abelian subgroups of $\PG(\F_q)$. Every such subgroup fixes\footnote{In the sense that every non-trivial element of this subgroup fixes the same element, and this element is the unique fixed point of every non-trivial element of the subgroup.} a unique element in $\F_q\cup\{\infty\}$ (see \cite[Section 4.2]{subpglk}). These groups are conjugated to the subgroups given by \[V':=\{x+v\mid v\in V\}\] for an $\F_p$-subspace $V\le\F_q$. Note that $d(V')=0$, since its unique fixed point is $\infty$, thus every other orbit is regular. If $G$ is a conjugated version of $V'$ in $\PG(\F_q)$, its fixed point may have changed as well. If the fixed point is in $\F_q$, then $d(G)=1$. Now we have everything ready to state Equation \eqref{eqcombination} for these two types of subgroups.

\begin{lem}\label{2types}
	Let $G\le\PG(\F_q)$ and $n>d(G)$.
	\begin{enumerate}
		\item If $G$ is cyclic of prime order $s$, then \[\omega_G(sn,1)=\omega_G(n,1)+(s-1)\omega_G(n,s).\]
		\item If $G$ is an elementary abelian subgroup of size $p^m$, then \[\omega_G(np,p^{m-1})=(p^m-1)\omega_G(n,p^m)+p^{m-1}\omega_G(n,p^{m-1}).\]
	\end{enumerate}
\end{lem}
\begin{proof}
	We start with the cyclic case. If $n>d(G)$, Equation \eqref{eqcombination} becomes \[0=\left(\omega_G(n,s)-s\omega_G(n,s)\right)+\left(\omega_G(ns,1)-\omega_G(n,1)\right),\] which is equivalent to what we state in the lemma above. What is left to show is that $\I_q^n\cap \I_q^G=\I_q^n$ for all $n>d(G)$. However, the only problematic case is $d(G)=0$, because the only $n$ for which the sets could possibly be not equal is $n=1$. If $d(G)=0$, then the sole fixed point of $G$ is $\infty$ by Lemma \ref{fix} and what we said after Equation \eqref{eqcombination}. Hence, $G(\infty)=\{\infty\}$ and therefore $\I_q^G=\I_q$.
	
	Now onto the case where $G$ is elementary abelian. We claim that, if $n>d(G)$, the only possible $G$-orbit sizes are $p^m$ and $p^{m-1}$. Let $r\in\I_q^n$ and consider the orbit $r^G$. By the orbit-stabilizer theorem we know that \[|r^G|=\frac{p^m}{|\Stab_G(r)|}.\] Every non-trivial cyclic subgroup of an elementary abelian subgroup is of prime order $p$. If $n>d(G)$, we know that the roots of $r$ in $\overline{\F}_q$ are contained in regular $G$-orbits. That is obvious for the case $d(G)=1$ again, because then $n\ge 2$ and the unique fixed point of $G$ is in $\F_q$. If $d(G)=0$, there are only regular $G$-orbits in $\overline{\F}_q$. With the help of Theorem \ref{galff} we know that $\Stab_G(r)$ has to be a cyclic subgroup of $G$ and $\PG(\F_q)$, hence either has order $1$ or $p$.
	
	Having established that the only orbit sizes are $p^m$ and $p^{m-1}$, Equation \eqref{eqcombination} becomes \begin{align*}
		0=\left(\omega_G(n,p^m)-p^m\omega_G(n,p^m)\right)+\left(\omega_G(np,p^{m-1})-p^{m-1}\omega_G(n,p^{m-1})\right).
	\end{align*} Rearranging this equation proves the lemma. The argument for the fact that $\I_q^n\cap\I_q^G=\I_q^n$ for all $n>d(G)$ is the same as for the cyclic case, because elementary abelian subgroups of $\PG(\F_q)$ have a unique fixed point. 
\end{proof}
\section{Recursive Behavior of Rational Transformations with Quotient Maps}\label{sec4}
Let $Q=\frac{g}{h}\in\F_q(x)$ be a rational function of degree $s$, define \begin{align}
	N(Q,n)&:=\{f\in\I_q^n\mid f^Q\text{ factors into exactly }s \text{ different irreducible polynomials of degree } n \}\label{yn1}\\
	Y(Q,n)&:=\{f\in\I_q^n\mid f^Q\text{ has an irreducible factor of degree greater than } n \}.\label{yn2}
\end{align} This definition is chosen so that we can handle the case of cyclic prime subgroups and elementary abelian subgroups with the same notation. One can define the sets differently, and in our opinion, more intuitively. However, then we would need to consider slightly different sets for both cases and we do not want to do that. The names for the two sets come from this more intuitive notation. The letter $N$ stands for ''No'' and $Y$ stands for ''Yes''. If we consider $Q$ to be a quotient map for a cyclic subgroup $G$ of prime order $s$, then $f^Q$, for irreducible $f$ of degree $n$, either factors into exactly $s$ irreducible polynomials of degree $n$ (this is a $G$-orbit on degree level $n$ of size $s$), or into an irreducible polynomial of degree $sn$ (this is a $G$-orbit on degree level $ns$ of size 1). In the first case, $f^Q$ is \textit{not} an irreducible polynomial and belongs to the set $N(Q_G,n)$. If $f^Q$ is irreducible, then it belongs to $Y(Q_G,n)$ and we obtained an irreducible polynomial of higher degree (hence the ''Yes''). Note that $Y(Q,n)$ and $N(Q,n)$ in general is not a partition of $\I_q^n$, but for $n>d(G)$ and $Q=Q_G$ a quotient map, it is a partition.

Before we handle the cyclic case and the elementary abelian case separately, we have to state a slightly different version of \textsc{Cohen}'s famous lemma about the irreducibility of rational transformations. \begin{lem}[{\cite[Lemma 1]{cohensLem}}]\label{cohenslem}
	Let $Q(x)=g(x)/h(x)\in K(x)$ and $F\in K[x]$ be a polynomial such that $F(Q(\infty))\neq 0$. The rational transformation $F^{Q}$ is irreducible if and only if $F\in K[x]$ is irreducible and $g(x)-\alpha h(x)$ is irreducible over $K(\alpha)[x]$, where $\alpha$ is any root of $F$.
\end{lem}
\begin{proof}
	The proof can be done as the original proof. The only thing remaining is the ''only if'' part that we added, that is, $F^Q$ is irreducible implies $F$ is irreducible and $g-\alpha h\in K(\alpha)[x]$ is irreducible. By Lemma \ref{transbasic}, item 4, $F$ has to be irreducible. Let $\gamma\in\overline{K}$ be a root of $F^Q$, then $Q(\gamma)=\alpha$ and, since $F^Q$ is irreducible, we have \[\deg(F^Q)=[K(\gamma):K]=[K(\gamma):K(\alpha)]\cdot[K(\alpha):K]=[K(\gamma):K(\alpha)]\cdot \deg(F).\]  
 Since $F(Q(\infty))\neq 0$, the degree of $F^Q$ is $\deg(F)\cdot\deg(Q)$ by Lemma \ref{transbasic} first item, hence $[K(\gamma):K(\alpha)]=\deg(Q)$. The polynomial $g-\alpha h\in K(\alpha)[x]$ has $\gamma$ as a root and is of degree $\deg(Q)$, thus is irreducible. \end{proof}
\subsection{The Cylic Case}
Let $G\le\PG(\F_q)$ be cyclic of prime order $s$ and $Q_G$ be a quotient map for $G$. Moreover, let $f\in\I_q^n$ with $n>d(G)$. As said before, $f^{Q_G}$ can factor in exactly two ways: \begin{enumerate}
	\item $f^{Q_G}\in N(Q_G,n)$ and its factors form a $G$-orbit in $\I_q^n$ of size $s$.
	\item $f^{Q_G}\in Y(Q_G,n)$ and it forms a $G$-orbit in $\I_q^{ns}$ of size $1$.
\end{enumerate}
Therefore, we can equate the cardinality of the sets $N(Q_G,n)$ and $Y(Q_G,n)$ with the number of orbits of size $s$ in $\I_q^n$ and of size $1$ in $\I_q^{ns}$ respectively. This is the main idea of the proof of the following theorem.
\begin{theorem}\label{cyclicyn}
	Let $G\le\PG(\F_q)$ be cyclic of prime order $s$ and $Q_G$ be a quotient map for $G$. For all $n>d(G)$ we have \[|Y(Q_G,n)|-(s-1)|N(Q_G,n)|=\begin{cases}
		0,&\text{ if }s\nmid n\\
		|Y(Q_G,n/s)|,&\text{ if }s\mid n.
	\end{cases}\]
\end{theorem}
\begin{proof}
	By Lemma \ref{2types} we have \[\omega_G(sn,1)=\omega_G(n,1)+(s-1)\omega_G(n,s)\Leftrightarrow \omega_G(sn,1)-(s-1)\omega_G(n,s)=\omega_G(n,1).\] By Corollary \ref{bij2}, and the fact that $s$ is prime, we get \begin{align*}
		\omega_G(sn,1)-(s-1)\omega_G(n,s)&=\omega_G(n,1)\\
		|Y(Q_G,n)|-(s-1)|N(Q_G,n)|&=\omega_G(n,1).
	\end{align*} The only thing remaining to show is that \begin{equation}\label{hardpart}\omega_G(n,1)=\begin{cases}
			0,&\text{ if }s\nmid n\\
			|Y(Q_G,n/s)|,&\text{ if }s\mid n.
	\end{cases}\end{equation}

	By Lemma \ref{n=d(G)} we have that the irreducible polynomials with roots in $P_G$ (as defined before Lemma \ref{fix}) are of degree $d(G)$. Since $n>d(G)$, Theorem \ref{galff} shows $\omega_G(n,1)=0$ if $s\nmid n$. 
	
	If $s\mid n$, the polynomial $F^{Q_G}$, for an irreducible polynomial $F\in\I_q^{n/s}$, factors in one of the following ways \begin{enumerate}
		\item $F^{Q_G}$ is irreducible of degree $n$, thus belongs to $Y(Q_G,n/s)$,
		\item $F^{Q_G}$ has $s$ different irreducible factors of degree $n/s$, thus belongs to $N(Q_G,n/s)$,
		\item $F^{Q_G}=(\prod r^G)^k$ with $k>1$.
	\end{enumerate}
	We need to examine the third case more closely because polynomials satisfying $F^{Q_G}=(\prod r^G)^k$, with $k>1$, could belong to $Y(Q_G,n/s)$. The orbit $r^G$ consists of a single irreducible polynomial $f$ whose roots are contained in $P_G$. To see this recall that $F^{Q_G}=(\prod r^G)^k$ with $k>1$ are the polynomials with roots in $P_G$ by Theorem \ref{main1}. Let $f=r$, which is an irreducible polynomial with a fixed point of $G$ as a root (which is an element of $P_G$ by Lemma \ref{fix}). The polynomial $f$ is $G$-invariant by Lemma \ref{rootsginv} and thus forms a $G$-orbit in $\I_q^G$, hence $r^G=\{f\}$. From Lemma \ref{n=d(G)} follows that $\deg(f)=\deg(F)=n/s$, showing that $F$ neither belongs to $Y(Q_G,n/s)$ nor to $N(Q_G,n/s)$. 
\end{proof}

\subsubsection{Square / Non-Square Conditions}
Let $p\neq 2$ and $a\in\F_q$. The Legendre symbol over $\F_q$ is defined as \begin{equation*}
	\left(\frac{a}{q}\right):=a^{(q-1)/2}=\begin{cases}
		0,&\text{ if } a=0\\
		1,&\text{ if } a \text{ is a square in }\F_q^{\ast}\\
		-1,&\text{ if }a \text{ is a non-square in }\F_q^{\ast}.
	\end{cases}
\end{equation*} For $u,v\in\F_q$ with $u\neq v$ and $n\in\mathbb{N}$ set \begin{align*}
	C_{u,v}(n)&:=\left\{f\in\I_q^n\mid f(u)\cdot f(v)\text{ is a non-square in }\F_q^{\ast}\right\}\\
	D_{u,v}(n)&:=\left\{f\in\I_q^n\mid f(u)\cdot f(v)\text{ is a square in }\F_q^{\ast}\text{ or }0\right\}.
\end{align*}
The following theorem holds. \begin{theorem}\label{squarenoncond}
	Let $q$ be odd. For all $u,v\in\F_q$ with $u\neq v$ and $n\ge 2$ we have \begin{equation*}
		|C_{u,v}(n)|-|D_{u,v}(n)|=\begin{cases}
			0,&\text{ if }2\nmid n\\
			|C_{u,v}(n/2)|,&\text{ if }2\mid n.
		\end{cases}
	\end{equation*}
\end{theorem}
\begin{proof}
	Consider the matrix \begin{equation*}
		A_{u,v}:=\left(\begin{array}{cc}
			\frac{1}{2}(u+v)& -uv\\
			1& -\frac{1}{2}(u+v)
		\end{array}\right).
	\end{equation*} Since $[A_{u,v}]^2=[I_2]$, the cyclic group $G_{u,v}:=\langle [A_{u,v}]\rangle\le \PG(\F_q)$ contains only 2 elements. As a quotient map for $G_{u,v}$ we choose \begin{equation*}
		Q_{G_{u,v}}(x)=\frac{1}{2}(x+[A_{u,v}](x))=\frac{x^2-uv}{2x-(u+v)}.
	\end{equation*} Note that $Q_{G_{u,v}}(u)=u$ and $Q_{G_{u,v}}(v)=v$, since $[A_{u,v}]\circ u=u$ and $[A_{u,v}]\circ v=v$. Moreover, $\{u\}$ and $\{v\}$ are the only non-regular $G_{u,v}$-orbits in $\overline{\F}_q\cup\{\infty\}$, hence $d(G_{u,v})=1$ by Lemma \ref{n=d(G)}.

	Let $F\in\I_q$, then $F^{Q_{G_{u,v}}}$ is irreducible if and only if \begin{align*}
		P(x)&:=(x^2-uv)-\gamma(2x-(u+v))\\
		&=x^2-2\gamma x+(\gamma(u+v)-uv)\in\F_q(\gamma)[x]
	\end{align*} is irreducible for $\gamma\in\overline{\F}_q$ a root of $F$ by Lemma \ref{cohenslem}. A quadratic polynomial over $\F_q(\gamma)$ is irreducible if and only if it has no roots in $\F_q(\gamma)$. For $P$, this is equivalent to $\gamma^2-(\gamma(u+v)-uv)$ being a non-square in $\F_q(\gamma)=\F_{q^{\deg(F)}}$. Hence \begin{align*}
		-1&=\left(\frac{\gamma^2-(u+v)\gamma+uv}{q^{\deg(F)}}\right)=\left(\frac{(\gamma-u)(\gamma-v)}{q^{\deg(F)}}\right)=\left((\gamma-u)(\gamma-v)\right)^{\frac{q^{\deg(F)}-1}{2}}\\&=\left((u-\gamma)(v-\gamma)\right)^{\frac{q^{\deg(F)}-1}{q-1}\cdot \frac{q-1}{2}}=\left(\prod\limits_{i=0}^{\deg(F)-1}(u-\gamma^{q^i})\right)^{\frac{q-1}{2}}\cdot \left(\prod\limits_{i=0}^{\deg(F)-1}(v-\gamma^{q^i})\right)^{\frac{q-1}{2}}\\&=F(u)^{\frac{q-1}{2}}\cdot F(v)^{\frac{q-1}{2}}=\left(\frac{F(u)\cdot F(v)}{q}\right).
	\end{align*} This calculation is based on a calculation trick that Meyn used in \cite[Proof of Theorem 8]{meyn}. We showed that \[Y(Q_{G_{u,v}},n)=\{f\in\I_q^n\mid \left(\frac{f(u)f(v)}{q}\right)=-1\},\] the rest of the proof follows from Theorem \ref{cyclicyn} using $N(Q_{G_{u,v}},n)=\I_q^n\setminus Y(Q_{G_{u,v}},n)=D_{u,v}(n)$ for $n\ge 2$. 
\end{proof}
\subsubsection{Power / Non-Power Conditions}  Let $\F_q$ be an arbitrary finite field, $s$ a prime dividing $q-1$ and $c\in\F_q$. For $n>1$ define \begin{align*}
	T_c(n)&:=\left\{f\in\I_q^n\mid  f(c)^{\frac{q-1}{s}}\neq 1\right\}\\
	U_c(n)&:=\I_q^n\setminus T_c(n)=\left\{f\in\I_q^n\mid f(c)^{\frac{q-1}{s}}= 1\right\}
\end{align*} and \begin{align*}
	T_c(1)&:=\left\{f\in\I_q^1\mid  f(c)^{\frac{q-1}{s}}\notin\{0,1\}\right\}\\
	U_c(1)&:=\I_q^n\setminus T_c(1)=\left\{f\in\I_q^1\mid f(c)^{\frac{q-1}{s}}\in\{0,1\}\right\}.
\end{align*}
\begin{theorem}\label{powercond}
	Let $\F_q$ be an arbitrary finite field and $s$ a prime dividing $q-1$. For all $n\ge 2$ and $c\in\F_q$ we have \begin{equation*}
		|T_c(n)|-(s-1)|U_c(n)|=\begin{cases}
			0,&\text{ if }s\nmid n\\
			|T_c(n/s)|,&\text{ if }s\mid n.
		\end{cases}
	\end{equation*}
\end{theorem}
Before we prove this theorem we need a lemma, which can be considered folklore. \begin{lem}\label{crit}
	Let $\F_q$ be an arbitrary finite field, $c\in\F_q^{\ast}$ and $s$ a prime dividing $q-1$. The polynomial $x^s-c\in\F_q[x]$ is irreducible if and only if $c^{\frac{q-1}{s}}\neq 1$. 
\end{lem}

\begin{proof}[Proof of Theorem \ref{powercond}.]
	We consider the following matrix \begin{equation*}
		A:=\left(\begin{array}{cc}
			a& b\\
			0& 1
		\end{array}\right) 
	\end{equation*} where $a\in\F_q^{\ast}$ has order $s>1$, which is prime, and $b\in\F_q$ is arbitrary. The group $G:=\langle [A]\rangle$ has order $s$. The fixed points of $[A]$ are $\infty$ and $c:=\frac{-b}{a-1}$, and these are again the only non-regular $\langle[A]\rangle$-orbits in $\overline{\F}_q\cup\{\infty\}$ by Lemma \ref{fix}. Note that we can obtain every $c\in\F_q$ for fixed $a\in\F_q^{\ast}\setminus\{1\}$ by choosing $b=(1-a)\cdot c$. A quotient map for $G$ is \begin{equation*}
		Q_G(x)=-(x-c)^s+c
	\end{equation*} and $(x-c)^{Q_G}=(x-c)^s$, so $d(G)=1$. 
	
	Let $F\in\I_q^n$, then $F^{Q_G}(x)=F(Q_G(x))$ is irreducible if and only if \begin{equation*}
		P(x)=Q_G(x)-\gamma=-(x-c)^s+c-\gamma\in\F_{q^n}[x] 
	\end{equation*} is irreducible by Lemma \ref{cohenslem}, where $\gamma$ is a root of $F$. The polynomial $P(x)\in\F_{q^n}[x]$ is irreducible if and only if $P(x+c)=-(x^s-(c-\gamma))\in\F_{q^n}[x]$ is irreducible. By Lemma \ref{crit}, this is the case exactly when $(c-\gamma)^{(q^n-1)/s}\neq 1$ (if $n=1$ we have to add the case of $c-\gamma=0$, so we need to adjust the condition to $(c-\gamma)^{(q^n-1)/s}\notin\{0,1\}$). Now we calculate: \begin{align*}
		1&\neq (c-\gamma)^{\frac{q^n-1}{s}}= (c-\gamma)^{\frac{q^n-1}{q-1}\cdot \frac{q-1}{s}}= \left(\prod\limits_{i=0}^{n-1}(c-\gamma^{q^i})\right)^{\frac{q-1}{s}}= F(c)^{\frac{q-1}{s}}.
	\end{align*} Hence $Y(Q_G,n)=T_c(n)$, the rest follows from Theorem \ref{cyclicyn} again. 
\end{proof}
\begin{bem}
	Taking the more natural quotient map $(x-c)^s+c$ in the proof of Theorem \ref{powercond} results in the conditions \begin{align*}
		T_c(n)&:=\left\{f\in\I_q^n\mid  f(c)^{\frac{q-1}{s}}\neq (-1^n)^{(q-1)/s}\right\}\\
		U_c(n)&:=\I_q^n\setminus T_c(n)=\left\{f\in\I_q^n\mid f(c)^{\frac{q-1}{s}}=(-1^n)^{(q-1)/s}\right\},
	\end{align*} for $n\ge 2$, and $T_c(1)=\left\{f\in\I_q^1\mid  f(c)^{\frac{q-1}{s}}\notin\{0,(-1^n)^{(q-1)/s}\} \right\}$. The cardinalities of these sets satisfy the same recursive equation.
\end{bem}
\subsection{The Elementary Abelian Case}\label{secuni}
Let $G\le\PG(\F_q)$ be an elementary abelian group of order $p^m$ and $Q_G$ be a quotient map. Moreover, for $f\in\I_q^n$ with $n>d(G)$, we know that $f^{Q_G}$ either has $p^m$ or $p^{m-1}$ irreducible factors by the proof of Lemma \ref{2types}. If the orbit size is $p^m$, then every irreducible factor of $f^{Q_G}$ has degree $n$, otherwise the orbit size is $p^{m-1}$ and the degree of every irreducible factor is $pn$. The following result holds. \begin{theorem}\label{elemabelian}
	Let $G\le\PG(\F_q)$ be an elementary abelian subgroup of order $p^m$ with quotient map $Q_G$. For all $n>d(G)$ we have \[|Y(Q_G,n)|-(p^m-1)|N(Q_G,n)|=\begin{cases}
		0,&\text{ if }p\nmid n\\
		p^{m-1}|Y(Q_G,n/p)|,&\text{ if }p\mid n.
	\end{cases}\]
\end{theorem}
\begin{proof}
	By Lemma \ref{2types} we have \begin{align*}&\omega_G(np,p^{m-1})=(p^m-1)\omega_G(n,p^m)+p^{m-1}\omega_G(n,p^{m-1})\\&\Leftrightarrow \omega_G(np,p^{m-1})-(p^m-1)\omega_G(n,p^m)=p^{m-1}\omega_G(n,p^{m-1}).\end{align*} As mentioned beforehand, polynomials in $Y(Q_G,n)$ are mapped to $G$-orbits in $\I_q^{np}$ of size $p^{m-1}$, and polynomials in $N(Q_G,n)$ are mapped to the regular orbits in $\I_q^n$ by $\delta_{Q_G}$ in Corollary \ref{bij2}. As a result \begin{align*}
		\omega_G(np,p^{m-1})-(p^m-1)\omega_G(n,p^m)&=p^{m-1}\omega_G(n,p^{m-1})\\
		|Y(Q_G,n)|-(p^m-1)|N(Q_G,n)|&=p^{m-1}\omega_G(n,p^{m-1}).
	\end{align*} The only thing remaining to show is that \begin{equation}\label{hardpart2}\omega_G(n,p^{m-1})=\begin{cases}
			0,&\text{ if }p\nmid n\\
			|Y(Q_G,n/p)|,&\text{ if }p\mid n.
	\end{cases}\end{equation} Note that orbits of size $p^{m-1}$ can exist on degree level $n$ only if $p\mid n$. Let $r^G$ be a $G$-orbit of size $p^{m-1}$ in $\I_q^n$. The stabilizer of every polynomial in this orbit has size $p$ and therefore has to be a cyclic subgroup of $G$ and of $\PG(\F_q)$. Since $n>d(G)$, we know that the $G$-orbits of every root of $r$ has to be regular, which means that the $\Stab_G(r)$-orbits are regular. By Theorem \ref{galff} we get $p\mid n$, showing $\omega_G(n,p^{m-1})=0$ if $p\nmid n$.
	
	Now let $p\mid n$. We want to show that $\omega_G(n,p^{m-1})=|Y(Q_G,n/p)|$. Consider an element $F\in\I_q^{n/p}$ and write \[F^{Q_G}=(\prod r^G)^k.\] What is left to show is that if $\deg(r)>\deg(F)$, i.e.\ $F\in Y(Q_G,n/p)$, then $|r^G|=p^{m-1}$ and $\deg(r)=n$. If $k=1$ this is clear, since then \[\deg(r)=|\Stab_G(r)|\cdot \deg(F)=|\Stab_G(r)|\cdot \frac{n}{p}\] by Equation \eqref{degreeformula}. So for $\deg(r)$ to be greater than $n/p$, the stabilizer of $r$ has to have cardinality $p$.
	
	If $k>1$ the roots of $F^{Q_G}$ are contained in non-regular $G$-orbits by Theorem \ref{main1}. If $d(G)=0$ there is no non-regular $G$-orbit. In case of $d(G)=1$, we know that $G$ has a unique fixed point in $\F_q$, say $a$. Thus $F^{Q_G}=(x-a)^k$ by Lemma \ref{rootsginv} and the same argumentation as in the proof of Theorem \ref{cyclicyn}. Additionally, $\deg(F)=1$ because the degree of the polynomial $F$ divides the degree of $x-a$. This means that $F\notin Y(Q_G,1)$.
	
\end{proof}
Theorem \ref{elemabelian} can be used to obtain the explicit numbers for $|Y(Q_G,n)|$ and $|N(Q_G,n)|$ if $n>d(G)$. \begin{theorem}\label{explicitnumelem}
	Let $G\le\PG(\F_q)$ be an elementary abelian subgroup of cardinality $p^m$ and let $Q_G$ be a quotient map for $G$. For all $n\in\mathbb{N}$ we have \[|Y(Q_G,n)|=\frac{p^m-1}{p^m n}\sum\limits_{\substack{d\mid n\\p\nmid d}}\mu(d)q^{n/d},\] where $\mu:\mathbb{N}\to \mathbb{N}$ is the Möbius function.
\end{theorem}
\begin{proof}
	Let $N_q(n):=|\I_q^n|$. For $n>d(G)$ the sets $Y(Q_G,n)$ and $N(Q_G,n)$ are a partition of $\I_q^n$, thus \[N_q(n)=|Y(Q_G,n)|+|N(Q_G,n)|.\] If $d(G)=1=n$, then this equation does not hold anymore. Let $a\in\F_q$ be such that $G(a)=\{a\}$. A partition of $\F_q\cup\{\infty\}$ into $G$-orbits is given by 
	\begin{align*}
		q+1&=|G(a)|+|G|\cdot (\text{Number of orbits of size }|G|)\\
		&=1+|G(\infty)|+|G|\cdot \omega_G(1,|G|).
	\end{align*} The second equality holds since $|G(\infty)|=|G|$ and all the other regular $G$-orbits $G(b)$, for $b\in\F_q$, can be associated with the $G$-orbits $(x-b)^G$ in $\I_q^1\cap \I_q^G$. Thus $\omega_G(1,|G|)=\frac{1}{|G|}(q-|G|)$. There are the following three possibilities for the factorization of $f^{Q_G}$ if $\deg(f)=1$ \begin{enumerate}
		\item $f^{Q_G}$ consists of $|G|/p$ irreducible polynomials of degree $p$, then $f\in Y(Q_G,1)$
		\item $f^{Q_G}$ splits into $|G|$ polynomials of degree 1, then $f\in N(Q_G,n)$
		\item $f^{Q_G}=(x-a)^{|G|}$.
	\end{enumerate} As a consequence, we have \begin{align*}
		q&=|N(Q_G,1)|+|Y(Q_G,1)|+1=\omega_G(1,|G|)+|Y(Q_G,1)|+1\\
		&=\frac{q-|G|}{|G|}+|Y(Q_G,1)|+1.
	\end{align*} After rearrangement we obtain $|Y(Q_G,1)|=\frac{|G|-1}{|G|}q=\frac{|G|-1}{|G|}N_q(1)$.

	If $n>d(G)$, Theorem \ref{elemabelian} holds. Combining both facts yields \begin{equation}\label{recursion}|Y(Q_G,n)|=\begin{cases}
			\frac{p^m-1}{p^m}N_q(n),&\text{if }p\nmid n\\
			\frac{p^m-1}{p^m}N_q(n)+\frac{1}{p}|Y(Q_G,n/p)|,&\text{if }p\mid n.
	\end{cases}\end{equation} We can include the case $n=1=d(G)$ in Equation \eqref{recursion} by what we wrote beforehand, because $|Y(Q_G,1)|=\frac{|G|-1}{|G|}N_q(1)$. Write $n=p^sa$ with $p\nmid a$ and $s\ge 0$. Since $d(G)\le 1$ for an elementary abelian subgroup of $\PG(\F_q)$, we have $p^l a\ge d(G)$ for all $0\le l\le s$. Applying Equation \eqref{recursion} repeatedly gives \begin{equation}
		|Y(Q_G,n)|=\frac{p^m-1}{p^m}\sum\limits_{i=0}^s\frac{1}{p^i}N_q(p^{s-i}a).
	\end{equation} Plugging in the exact values $N_q(p^{s-i}a)=\frac{1}{p^{s-i}a}\sum_{d\mid p^{s-i}a}\mu(d)q^{p^{s-i}a/d}$ gives \[|Y(Q_G,n)|=\frac{p^m-1}{p^m n}\sum\limits_{i=0}^s\sum\limits_{d\mid p^{s-i}a}\mu(d)q^{p^{s-i}a/d}.\] If $i\neq s$, we can simplify the condition under the second sum to $d\mid pa$, since all the $d$ divisible by $p^2$ are mapped to $0$ by $\mu(\cdot)$. Therefore \[|Y(Q_G,n)|=\frac{p^m-1}{p^m n}\left(\sum\limits_{d\mid a}\mu(d)q^{a/d}+\sum\limits_{i=0}^{s-1}\sum\limits_{d\mid p a}\mu(d)q^{p^{s-i}a/d}\right).\] The condition $d\mid pa$ can be split into two cases, namely \begin{enumerate}
		\item $d\mid a$
		\item $p\mid d$ and $d/p\mid a$.
	\end{enumerate} In the latter we write $j=d/p$. Note that $p\nmid j$, therefore we can rewrite the second part of $|Y(Q_G,n)|$ as \begin{align*}
		\sum\limits_{i=0}^{s-1}\sum\limits_{d\mid p a}\mu(d)q^{p^{s-i}a/d}&=\sum\limits_{i=0}^{s-1}\left(\sum\limits_{d\mid a}\mu(d)q^{\frac{a}{d}p^{s-i}}+\sum\limits_{j\mid a}\mu(pj)q^{\frac{a}{j}p^{s-i-1}}\right)\\&=\sum\limits_{i=0}^{s-1}\left(\sum\limits_{d\mid a}\mu(d)q^{\frac{a}{d}p^{s-i}}-\sum\limits_{j\mid a}\mu(j)q^{\frac{a}{j}p^{s-i-1}}\right)\\&=\sum\limits_{d\mid a}\mu(d)q^{n/d}-\sum\limits_{j\mid a}\mu(j)q^{a/j}.
	\end{align*} As a result we obtain \[|Y(Q_G,n)|=\frac{p^m-1}{p^m n}\left(\sum\limits_{d\mid a}\mu(d)q^{a/d}+\sum\limits_{d\mid a}\mu(d)q^{n/d}-\sum\limits_{j\mid a}\mu(j)q^{a/j}\right)=\frac{p^m-1}{p^m n}\sum\limits_{\substack{d\mid n\\p\nmid d}}\mu(d)q^{n/d}.\] 
\end{proof}
\begin{bem}
	The number of elements in $N(Q_G,n)$ is given by $|N(Q_G,n)|=N_q(n)-|Y(Q_G,n)|$ if $n>d(G)$, hence \begin{align*}
		|N(Q_G,n)|&=\frac{1}{n}\left(\sum\limits_{d\mid n}\mu(d)q^{n/d}-\frac{p^m-1}{p^m}\sum\limits_{\substack{d\mid n\\p\nmid d}}\mu(d)q^{n/d}\right)\\&=\frac{1}{n}\left(\sum\limits_{\substack{d\mid n\\p\mid d}}\mu(d)q^{n/d}+\frac{1}{p^m}\sum\limits_{\substack{d\mid n\\p\nmid d}}\mu(d)q^{n/d}\right).
	\end{align*}
\end{bem}
\begin{bem}\label{explicit numbers}
	In the same vain as Theorem \ref{explicitnumelem}, one can use Theorem \ref{cyclicyn} to obtain the explicit number for the cyclic case if the group order is a prime number. However, the number of elements in $A_q(n):=|\{f\in\I_q^n\mid f^{Q_G}\text{ is irreducible}\}|$ is known for any cyclic subgroup $G\le\PG(\F_q)$, see \cite[Theorem 1.4.]{reisEx}. Recall that \[A_q(n)=\text{ Number of }G\text{-invariant irreducible polynomials of degree }n|G|.\]
\end{bem}
\subsubsection{Unipotent Subgroups and Generalized Trace Conditions}\label{unisec}
Let $V\le_p\F_q$ be an $\F_p$-subspace of $\F_q$. The associated subspace polynomial is given by \[Q_V(x)=\prod\limits_{v\in V}(x-v)\] and it is quotient map of the group $G_V:=\{x+v\mid v\in V\}\le \PG(\F_q)$. Since $G_V\cong \F_p^m$ for some $m\in\mathbb{N}$, we are in the elementary abelian case and Theorem \ref{elemabelian} holds for the sets $Y(Q_V,n)$ and $N(Q_V,n)$. Recall that $d(G_V)=0$.

To describe these sets in an arithmetical way we need to state a lemma first. \begin{lem}[{\cite[Chapter 2, Theorem 1]{orecontrib}} or {\cite[Theorem 10.5]{bluher1}}]\label{exactseq}
	Let $V\le_p\F_q$ be an $\F_p$-subspace with associated subspace polynomial $Q_V\in\F_q[x]$. Set $W:=Q_V(\F_q)\le_p\F_q$ with subspace polynomial $Q_W\in\F_q[x]$. Then \[x^q-x=Q_W\circ Q_V=Q_V\circ Q_W.\]
\end{lem} Based on that lemma, we define the \textit{complement sequence} of a subspace polynomial $Q_V$ of an $\F_p$-subspace $V\le \F_q$ as the sequence of polynomials $(Q_V^{(n)})_{n\in\mathbb{N}}$ such that for every $n$ we have \[Q_V^{(n)}\circ Q_V=x^{q^n}-x=Q_V\circ Q_V^{(n)}.\] Notice that every polynomial in the complement sequence is a polynomial in $\F_q[x]$. To see this consider a root $\alpha\in\F_{q^n}$ of $Q_V^{(n)}$. As a consequence of Lemma \ref{exactseq} we get that $\ker(Q_W)=\operatorname{im}(Q_V)$, where both polynomials are interpreted as mappings from $\F_{q^n}$ to itself. Thus there is $\beta\in\F_{q^n}$ such that $Q_V(\beta)=\alpha$, which shows that $\alpha^q=Q_V(\beta^q)$ is also a root of $Q_V^{(n)}$, hence $Q_V^{(n)}\in\F_q[x]$.

The sets $Y(Q_V,n)$ and $N(Q_V,n)$ can be described in terms of the complement sequence of $Q_V$. \begin{lem}\label{gentracecon}
	Let $V\le_p\F_q$ with subspace polynomial $Q_V$ and let $(Q_V^{(n)})_{n\in\mathbb{N}}$ be its complement sequence. The following equalities hold \begin{align*}
		Y(Q_V,n)&=\{f\in\I_q^n: f\nmid Q_V^{(n)}\}\\
		N(Q_V,n)&=\{f\in\I_q^n: f\mid Q_V^{(n)}\}.
	\end{align*}
\end{lem}\begin{proof}
	Consider the composition $f\circ Q_V$, where $f\in\I_q^n$. If $f\circ Q_V$ factors into a $G_V$-orbit of size $|V|$, then it only has irreducible polynomials of degree $n$ as factors. Let $\alpha\in\overline{\F}_q$ be a root of $f$ and $\gamma\in \overline{\F}_q$ be a root of $f\circ Q_V$ such that $Q_V(\gamma)=\alpha$. If $f\in N(Q_V,n)$, then $\gamma\in\F_{q^n}=\F_q(\alpha)$. Since $\alpha\in\operatorname{im}(Q_V)$ we know $\alpha\in\ker(Q^{(n)}_V)$ by Lemma \ref{exactseq}, showing $N(Q_V,n)\subseteq \{f\in\I_q^n\mid f\mid Q_V^{(n)}\}$. However, the whole argument can be reversed, i.e.\ if $f\mid Q_V^{(n)}$, then there is $\gamma\in\F_{q^n}=\F_q(\alpha)$ such that $Q_V(\gamma)=\alpha$, for a root $\alpha$ of $f$, and thus $f\circ Q_V$ has only degree $n$ irreducible polynomials as irreducible factors. Consequently $N(Q_V,n)\supseteq \{f\in\I_q^n\mid f\mid Q_V^{(n)}\}$, showing equality. 
\end{proof} With that, we can formulate the main theorem of this subsection. As in the introduction, define \begin{align*}
	S_V(n)&:=\{f\in\I_q^n: f\nmid Q_V^{(n)}\}\\
	T_V(n)&:=\{f\in\I_q^n: f\mid Q_V^{(n)}\}.
\end{align*} \begin{theorem}\label{unipotent}
	Let $V\le_p\F_q$ with subspace polynomial $Q_V$ and complement sequence $(Q_V^{(n)})_{n\in\mathbb{N}}$. For all $n\in\mathbb{N}$ we have \begin{equation}
		|S_V(n)|-(|V|-1)|T_V(n)|=\begin{cases}
			0,&\text{if }p\nmid n\\
			\frac{|V|}{p}|S_V(n/p)|,&\text{if }p\mid n.
		\end{cases}
	\end{equation}
\end{theorem}
\begin{proof}
	Is a direct consequence of Theorem \ref{elemabelian} and Lemma \ref{gentracecon}. 
\end{proof}
We will prove Equation \ref{gengranger} now. Let $V:=\F_{p^m}\le \F_q=\F_{p^t}$, so $m\mid t$. The quotient map/subspace polynomial of $G_V$ is $Q_V(x)=x^{p^m}-x$. It is a well-known fact that \[\Tr_{q^n/p^m}(x)\circ Q_V(x)=x^{q^n}-x=Q_V(x)\circ\Tr_{q^n/p^m}(x),\] hence the complement sequence consists of the polynomials \[Q_V^{(n)}(x)=\sum\limits_{i=0}^{n\frac{t}{m}-1}x^{p^{mi}}=\Tr_{q^n/p^m}(x).\] For $f\in\I_q^n$ and a root $\alpha$ of $f$ define $\Tr_{p^m}(f):=-\Tr_{q^n/p^m}(\alpha)$. Recall that \[-\Tr_{q^n/p^m}(\alpha)=\Tr_{q/p^m}(-\Tr_{q^n/q}(\alpha))=\Tr_{q/p^m}(a_{n-1}(f)).\] Let $a\in\F_{p^m}$, we set \[S_a^m(n):=\{f\in\I_q^n\mid \Tr_{p^m}(f)=a\}.\] 
\begin{theorem}\label{equation4}
	Let $V:=\F_{p^m}\le\F_q$ with subspace polynomial $Q_V=x^{p^m}-x$. For all $n\in\mathbb{N}$ and $a\in\F_{p^m}^{\ast}$ we have \begin{equation*}
		|S_a^m(n)|-|S_0^m(n)|=\begin{cases}
			0,&\text{if } p\nmid n\\
			p^{m-1}|S_a^m(n/p)|,&\text{if } p\mid n.
		\end{cases}
	\end{equation*}
\end{theorem}
\begin{proof} As explained in the introduction, we have \begin{align*}S_V(n)&=\{f\in\I_q^n: f\nmid\Tr_{q^n/p^m}(x)\}\\&=\{f\in\I_q^n: \Tr_{q^n/p^m}(\alpha)\neq 0\text{ for a root }\alpha\text{ of }f\}\\&=\{f\in\I_q^n: -\Tr_{q/p^m}(a_{n-1}(f))\neq 0\}\\&=\bigcup\limits_{a\in\F_{p^m}^{\ast}}S_a^m(n)\end{align*} and $S_0^m(n)=T_V(n)$.
	As a consequence of Theorem \ref{unipotent} we obtain \begin{equation}\label{beforedivide}\sum_{a\in\F_{p^m}^{\ast}}|S_a^m(n)|-(p^m-1)|S_0^m(n)|=\begin{cases}
			0,&\text{if }p\nmid n\\
			p^{m-1}\sum_{a\in\F_{p^m}^{\ast}}|S_a^m(n/p)|,&\text{if }p\mid n.
	\end{cases}\end{equation}
	Let $a,b\in\F_{p^m}$ be such that $ab\neq 0$. We show that the sets $S_a(n)$ and $S_b(n)$ have the same number of elements. We claim the map \[f(x)\mapsto \left(\frac{b}{a}\right)^nf(\frac{a}{b}x)\] is a bijection between $S_a^m(n)$ and $S_b^m(n)$. Let $\alpha\in\overline{\F}_q$ be a root of $f\in S_a^m(n)$, then $\frac{b}{a}\alpha$ is a root of $(\frac{b}{a})^nf(\frac{a}{b}x)$. Moreover, we have \begin{align*}
		\Tr_{p^m}((\frac{b}{a})^nf(\frac{a}{b}x))&=-\Tr_{q^n/p^m}(\frac{b}{a}\alpha)=\frac{b}{a}\cdot(-\Tr_{q^n/p^m}(\alpha))=\frac{b}{a}\cdot a=b
	\end{align*} and the map where the roles of $a$ and $b$ are swapped is the inverse mapping. Thus \[\sum_{b\in\F_{p^m}^{\ast}}|S_b^m(n)|=(p^m-1)\cdot |S_a^m(n)|\] and the same is true for the sum of $S_b^m(n/p)$ if $p\mid n$. As a consequence we can divide Equation \eqref{beforedivide} by $(p^m-1)$ and get \[	|S_a^m(n)|-|S_0^m(n)|=\begin{cases}
		0,&\text{if } p\nmid n\\
		p^{m-1}|S_a^m(n/p)|,&\text{if } p\mid n.
	\end{cases}\] 
\end{proof}
\begin{bem}\label{rem30} The most literal generalization of Equation \ref{gengranger} is the edge case $m=\log_p(q)$ in Theorem \ref{equation4}, since then \begin{align*}
		S_a^{\log_p(q)}(n)&=\{f\in\I_q^n\mid a_{n-1}(f)=a\}.
	\end{align*}
	The explicit numbers for $|S_a^1(n)|$ and $|S_a^{\log_p(q)}(n)|$ are known in the literature. For the first we refer the reader to Remark \ref{explicit numbers}, for the latter see \cite{carlitztrace, tracemoeb} and \cite{yucas}. All of these match with Theorem \ref{explicitnumelem}.
	
	The recursive equation for $S_a^1(n)$ can also be obtained by applying Theorem \ref{cyclicyn} to the subgroup $G=\{x+i\mid i\in\F_p\}$ and its quotient map $x^p-x$.
\end{bem}

\section{Conjugating Subgroups and Equivalent Quotient Maps}
Two rational functions $Q,R\in\F_q(x)$ are \textit{equivalent} $R\sim Q$ if there are $[A],[B]\in\PG(\F_q)$ such that $R=[A]\circ Q\circ [B]$. Notice that $Q$ is a generator of the field $\F_q(x)^G$, for a subgroup $G\le\PG(\F_q)$, if and only if $R=[A]\circ Q\circ [B]$ is generator of $[B]^{-1}\circ G\circ [B]$. The number of irreducible polynomials obtained by transformation $F^R$ is the same as by $F^Q$ if $R\sim Q$ and $\deg(F^R)>1$, see \cite[Lemma 2]{cubictrafo}. Moreover, assume $F^Q=\prod r_i$, then $(F^{[A]^{-1}})^R$ factors in the following way\begin{align}\label{conjeq}
	(F^{[A]^{-1}})^R=(F^{[A]^{-1}})^{[A]\circ Q\circ [B]}=(F^Q)^ {[B]}=(\prod r_i)^{[B]}=\prod r_i^{[B]}
\end{align}by Lemma \ref{transbasic}, item 3 and 6, if $F([A]^{-1}(\infty))\neq 0$ and $F(Q(\infty))\neq 0$. Additionally, if $r_i([B](\infty))\neq 0$, the factorization patterns of $F^Q$ and $(F^{[A]^{-1}})^R$ are the same, which means that $|N(R,n)|=|N(Q,n)|$ and $|Y(R,n)|=|Y(Q,n)|$ (as in \eqref{yn1} and \eqref{yn2}) if $n>1$. Therefore, if $Q=Q_G$ is a quotient map for a cyclic subgroup of prime order $s$, or an elementary abelian subgroup, then all to $Q$ equivalent rational functions $R$ satisfy Theorem \ref{cyclicyn} and \ref{elemabelian} for $n> \max\{d(G),1\}$.

We want to explore what equivalence does to the arithmetical conditions describing $N(Q,n)$ and $Y(Q,n)$. Let $F\in\I_q^n$ with $n\ge 2$ and let $\alpha$ be a root of $F$. Moreover, let $Q_G$ be a quotient map for $G$. The describing condition for $N(Q_G,n)$ comes from \textsc{Cohen}'s Lemma \ref{cohenslem} by solving the equation \[Q_G(x)=\alpha.\] Similarly, for $R=[A]\circ Q_G\circ [B]$, the condition for $F^R$ to be irreducible (or factor completely) comes from solving \begin{equation}\label{descreq}R(x)=([A]\circ Q_G\circ [B])(x)=[A](Q_G([B](x)))=\alpha.\end{equation} The polynomial $F$ is in $N(R,n)$ if all the roots of $F^R$ are in $\F_q(\alpha)=\F_{q^n}$, i.e.\ the equation above only has solutions in $\F_{q^n}$. Note that $\gamma\in\F_{q^n}$ is a solution to $[A]\circ Q_G(y)=\alpha$ if and only if $[B]^{-1}(\gamma)$ is a solution to \eqref{descreq} and, which is important, $[B]^{-1}(\gamma)\in\F_{q^n}$ if $n>1$. This allows us to use the substitution $y=[B](x)$ in Equation \eqref{descreq}. Consequently, we obtain \begin{equation}\label{newcond}Q_G(y)=[A]^{-1}(\alpha).\end{equation} If $n=1$ and $R=g/h$ with $\deg(g)\le \deg(h)$, there is the possibility that $\infty$ comes into play, i.e.\ $\gamma=\infty$ is a solution to \eqref{descreq} or $[B]^{-1}(\gamma)=\infty$ is a solution to \eqref{newcond}. This is the reason why $n=1$ has to be treated separately. Another reason is that the defining conditions may differ slightly for $n=1$ and $n>1$.

To showcase how that works, we consider an example which will give us a generalization of Theorem \ref{squarenoncond}. Let $s$ be a prime number and consider a primitive $s$-th root of unity in $\F_q$ (thus $s\mid q-1$). The group \[G=\{a^ix\mid i\in\mathbb{N}\}\le\PG(\F_q)\] is of order $s$ and a quotient map is $x^s$. Let $u\neq v$ be arbitrary elements in $\F_q$. We consider the following conjugated subgroup \[U:=\left[\begin{pmatrix}
	u&v\\
	1&1
\end{pmatrix}\right]G\left[\begin{pmatrix}
	1&-v\\
	-1&u
\end{pmatrix}\right]=\langle\left[\begin{pmatrix}
	v-au&uv(a-1)\\
	1-a&-u+av
\end{pmatrix}\right]\rangle.\] The rational function \begin{equation}\label{newquot}
	Q_U(x)=\left[\begin{pmatrix}
		-u&v\\
		-1&1
	\end{pmatrix}\right]\circ x^s\circ \left[\begin{pmatrix}
		1&-v\\
		-1&u
	\end{pmatrix}\right]=\frac{v(-x+u)^s-u(x-v)^s}{(-x+u)^s-(x-v)^s}
\end{equation} is a generator of $\F_q(x)^U$ and is equivalent to $x^s$ (if $s>2$ it is not a quotient map). Let $\alpha$ be a root of the irreducible polynomial $F\in\I_q^n$. Applying Equation \eqref{newcond} to this setup we get \[y^s=\left(\frac{\alpha-v}{\alpha-u}\right),\] where $\alpha\neq u$. From Lemma \ref{crit} follows that this has a solution in $\F_{q^n}$ (if $n>1$) if and only if \[\left(\frac{\alpha-v}{\alpha-u}\right)^{(q^n-1)/s}=1=\left(\frac{v-\alpha}{u-\alpha}\right)^{(q^n-1)/s}.\] By repeating the ''Meyn-trick'' (as in the proof of Theorem \ref{powercond}) we get \begin{equation*}
	F(u)^{(q-1)/s}=F(v)^{(q-1)/s}.
\end{equation*} For $n\ge 2$ define \begin{align}\label{uvgencond}
	C_{u,v}^s(n)&=\{f\in\I_q^n\mid f(u)^{(q-1)/s}\neq f(v)^{(q-1)/s}\}\\
	D_{u,v}^s(n)&=\{f\in\I_q^n\mid f(u)^{(q-1)/s}= f(v)^{(q-1)/s}\}.
\end{align} For $n=1$ we have \begin{align*}
	&Q_U(U(\infty))=\begin{cases}
		\infty,&\text{if }s=2\\
		\frac{u+v}{2},&\text{if }s>2
	\end{cases}\\
	&Q_U(u)=u\\
	&Q_U(v)=v\\
	&Q_U(U(\frac{u+v}{2}))=\infty.
\end{align*} Note that $(u+v)/2\in U(\infty)$ if $s=2$. Define \[C_{u,v}^s(1):=\{f\in\I_q^1\setminus\{x-u,x-v,x-\frac{u+v}{2}\}\mid f(u)^{(q-1)/s}\neq f(v)^{(q-1)/s}\}.\] We can exclude the polynomials $x-u,x-v$ and $x-\frac{u+v}{2}$ without concern, since \begin{align*}
	&(x-u)^{Q_U}=(x-u)^s\\
	&(x-v)^{Q_U}=(x-v)^s\\
	&(x-\frac{u+v}{2})^{Q_U}=\text{A polynomial of degree }s-1\text{ splitting completely into linear factors if }s>2,
\end{align*} and the polynomial $x-\frac{u+v}{2}$ does not satisfy the condition $f(u)^{(q-1)/2}\neq f(v)^{(q-1)/2}$ if $s=2$. All of this means that these three polynomials do not belong to $Y(Q_U,1)$. We obtain the following generalization of Theorem \ref{squarenoncond}. \begin{theorem}\label{gensquarecond}
	Let $u\neq v\in\F_q$ and $s$ be a prime dividing $q-1$. For all $n>1$ we have \begin{equation*}
		|C_{u,v}^s(n)|-(s-1)|D_{u,v}^s(n)|=\begin{cases}
			0,&\text{if }s\nmid n\\
			|C_{u,v}^s(n/s)|,&\text{if }s\mid n.
		\end{cases}
	\end{equation*}
\end{theorem}


\printbibliography

\end{document}